\documentclass[12pt,twoside,reqno]{amsart}
\usepackage[colorlinks=false,citecolor=blue]{hyperref}
\usepackage{mathptmx, amsmath, amssymb, amsfonts, amsthm, mathptmx, enumerate, color,mathrsfs}
\usepackage{graphicx}
\newcommand{\vertiii}[1]{{\left\vert\kern-0.25ex\left\vert\kern-0.25ex\left\vert #1 
    \right\vert\kern-0.25ex\right\vert\kern-0.25ex\right\vert}}

\usepackage{epstopdf}
\newtheorem{theorem}{Theorem}[section]
\newtheorem{lemma}{Lemma}[section]
\newtheorem{remark}{Remark}[section]

\newtheorem{corollary}{Corollary}[section]

\numberwithin{equation}{section}

\begin{document}
%inequalities involving
%\title{Norm and Spectral radius inequalities for operators on Hilbert space}
\title{Further bounds on $p$-numerical radii of operators via generalized Aluthge transform}

\author{Satyajit Sahoo}
\address{(Satyajit Sahoo) Department of Mathematics, 
	School of Basic Sciences, Indian Institute of Technology Bhubaneswar, Odisha 752050, India.}
\email{ssahoomath@gmail.com}

\begin{abstract}
The main aim of this article is to establish several $p$-numerical radius inequalities via the $(f,g)$-Aluthge transform of Hilbert space operators and operator matrices. Furthermore, various classical numerical radius and norm inequalities for Hilbert space operators are also discussed. The bounds obtained in this work improve upon several well-known earlier results. 
\end{abstract}
\maketitle
%------------------------------------------------------------------------------------%
\pagestyle{myheadings}
\markboth{\centerline{}}
{\centerline{}}
\bigskip
\bigskip
%------------------------------------------------------------------------------------%
\section{Introduction}
The numerical radius is a concept in functional analysis and operator theory, primarily used to measure the ``size" of operators in a Hilbert space. It is defined for a bounded linear operator 
$T$ on a Hilbert space 
$\mathcal{H}$.	Let $(\mathcal{H},\langle \cdot, \cdot \rangle)$ be a separable complex Hilbert space, and denote by $\mathfrak{L}(\mathcal{H})$ the space of all bounded linear operators on $\mathcal{H}$. The modulus of $T$ is given by $|T| = (T^*T)^{1/2}$. For $T\in \mathfrak{L}(\mathcal{H})$, let $T=\Re(T)+i\Im(T)$ be the Cartesian decomposition of $T$, while its real and imaginary parts are defined as $\Re (T) = \frac{T + T^*}{2}$ and $\Im(T) = \frac{T - T^*}{2i}$, respectively. An operator $T$ is called positive, denoted $T \geq 0$, if $\langle Tx, x \rangle \geq 0$ for all $x \in \mathcal{H}$. An operator $T$ is self-adjoint (or Hermitian) if $T = T^*$. The set of positive operators forms a convex cone in $\mathfrak{L}(\mathcal{H})$, inducing the partial order $\geq$ on the set of self-adjoint operators: for Hermitian operators $A$ and $B$, we write $A \geq B$ if and only if $A - B \geq 0$. This order is known as the L\"{o}wner order. It is evident that $|T| \geq 0$, and that $\Re(T)$ and $\Im(T)$ are self-adjoint for any operator $T \in \mathfrak{B}(\mathcal{H})$. 		
		An operator $T$ is normal if it satisfies $T^*T = TT^*$, and it is unitary if $T^*T = TT^* = I$.

Let $\mathfrak{K}(\mathcal{H})$ denote the ideal of compact operators on the Hilbert space $\mathcal{H}$.  For $T\in\mathfrak{K}(\mathcal{H})$ the singular values of $T$, denoted by $s_1(T), s_2(T), \dots$, correspond to the eigenvalues of the positive operator $|T|$ and are arranged in decreasing order, accounting for multiplicity. Furthermore, let 
		\begin{equation*}
			\mathcal{S}:=\left\{T\in\mathfrak{K}(\mathcal{H}):\, \sum_{j=1}^\infty s_j(T)<\infty
			\right\}.
		\end{equation*}
		Operators in $\mathcal{S}$ are called the trace class operators.
		The trace functional, denoted by $\text{tr}(\cdot)$, is defined on $\mathcal{S}$ as
		\begin{equation}\label{traza}
			\text{tr}(T) = \sum_{j=1}^{\infty} \langle Te_j, e_j\rangle,\quad T\in\mathcal{S},
		\end{equation}
		where $\{e_j\}_{j=1}^{\infty}$ forms an orthonormal basis for the Hilbert space $\mathcal{H}$. It is worth noting that this definition coincides with the standard trace definition when $\mathcal{H}$ is finite-dimensional. The series in \eqref{traza} converges absolutely, and its value remains unchanged regardless of the choice of basis.
		
		Additionally, let us clarify the definition of the Schatten $p$-class with $p \geq 1$. An operator $T$ belongs to the Schatten $p$-class, denoted as $\mathfrak{C}_p(\mathcal{H})$, if the sum of the $p$-th powers of its singular values is finite. More precisely, $T \in \mathfrak{C}_p(\mathcal{H})$ if 
		
		\[
		\text{tr}(|T|^p) = \sum_{j=1}^{\infty} s_j(T)^p < \infty.
		\]
		The Schatten $p$-norm of $T \in \mathfrak{C}_p(\mathcal{H})$ is given by
		\begin{equation*}
			\|T\|_p := \left[\text{tr}(|T|^p)\right]^\frac{1}{p}.
		\end{equation*}
         We remark here that for $p=\infty$, the Schatten $p$-norm is usual operator norm $\|T\|=\sup\limits_{\|x\|=1}\|Tx\|.$ When $p=2$, the ideal $\mathfrak{C}_2(\mathcal{H})$ is referred to as the Hilbert--Schmidt class. In this case, $\mathfrak{C}_2(\mathcal{H})$ forms a Hilbert space with the inner product $\langle T, S\rangle_2 = \text{tr}(TS^*)$. Also, when $p=1$, we obviously have that $\mathfrak{C}_1(\mathcal{H})=\mathcal{S}$.
		
For every $T, S\in \mathfrak{C}_p(\mathcal{H}), 0<p\leq \infty$, we obtain the following relations:
\begin{align*}
    \|T\|_{rp}^r=\||T|^r\|_p=\||T^*|\|_p ~\mbox{for}~r>0,
\end{align*}
and 
\begin{align}\label{Eq_1.2}
    \left\|\begin{bmatrix}
        T &0\\
        0 & S    \end{bmatrix}\right\|_p=\left\|\begin{bmatrix}
        0 &T\\
        S & 0
    \end{bmatrix}\right\|_p=
\left\{ \begin{array}{lll}
(\|T\|_p^p+\|S\|_p^p)^{1/p} & \textnormal{for $ 0<p<\infty,$}\\
\max\big\{\|T\|, \|S\|\big\}& \textnormal{for $ p=\infty$}. 
\end{array} \right.
\end{align}

For $1\leq p\leq q\leq \infty$, the Schatten $p$-norm of $T$ satisfies the monotonicity property 
$$\|T\|_\infty\leq \|T\|_q\leq \|T\|_p\leq \|T\|_1.$$
Moreover, if $T\in \in \mathfrak{C}_p(\mathcal{H})$ and $S\in\mathfrak{L}(\mathcal{H})$, then
\begin{align}\label{Eq_2.2}
    \|TS\|_p\leq \|T\|_p\|S\|~~\mbox{and}~\|ST\|_p\leq \|S\|\|T\|_p.
\end{align}
For $1\leq p\leq \infty$, the $p$-numerical radius of $T\in \mathfrak{C}_p(\mathcal{H})$ is defined by 
$$w_p(T)=\sup_{\theta\in\mathbb{R}}\|\Re(e^{i\theta}T)\|_p.$$
Similarly, $w_p(.)$ is defined for the case $0<p<1$. Note that $w_p(.)$ is weakly unitarily invariant i.e. $w_p(UTU^*)=w_p(T)$ for every unitary operator $U\in \mathfrak{L}(\mathcal{H})$ and for every $T\in \mathfrak{C}_p(\mathcal{H})$. It is known that for $1\leq p\leq \infty$, we obtain the following inequality 
\begin{align}
\frac{1}{2}\|T\|_p\leq w_p(T)\leq \|T\|_p.
\end{align}
If $T$ is self-adjoint, then $w_p(T)=\|T\|_p$. For more details, the reader may (see\cite{Benmakhlouf_Hirzallah_Kittaneh_LAMA_2021, Frakis_Kittaneh_Soltani_JAMC_2024, Frakis_Kittaneh_Soltani__2024, Frakis_Kittaneh__2024, Bhunia_Sahoo_2025}) and the references therein. 
The authors in \cite{AAbuOmarFKittanehLAA2019} showed some properties of $w_2(\cdot)$ that come along with those of $w(\cdot).$ For example, they showed in the same reference that if $T\in \mathfrak{C}_2$, then
\begin{align}\label{intr105}
   w_2(T)=\sqrt{\frac{1}{2}\|T\|_2^2+\frac{1}{2}|tr T^2|},
\end{align}
which implies 
\begin{align}\label{intr106}
    w_2(T)=\frac{1}{\sqrt{2}}\|T\|_2,\quad{\text{if}}\;T^2=0.
\end{align} 
There are other results involving classical numerical radius and Hilbert-Schmidt numerical radius have been established in \cite{Aici_Frakis_Kittaneh_Acta_2023, Aici_Frakis_Kittaneh_NFAO_2023, Aici_Frakis_Kittaneh_RCMP_2023, Aldalabih_Kittaneh_LAA_2019, SND, SND1, Sababheh_Moradi_Sahoo_LAMA_2024, Sababheh_Djordjević_MoradiCAOT_2024, Sahoo_Rout_AIOT_2022, Sahoo_Moradi_Sababheh_ACTA_Sz_math_2024, Sahoo_Sababheh_Filomat_2021, Zamani_LAA_2023, Bhunia_Kittaneh_Sahoo_LAA_2025} and the references therein.  

Let $T=U|T|$ be the polar decomposition of $T$. Then $T^*=U^*|T^*|$ is the polar decomposition of $T^*$. The Aluthge
transform of the operator $T$, denoted as $\widetilde{T}$ is
defined as $\widetilde{T}= |T|^{\frac{1}{2}}U|T|^{\frac{1}{2}}$. This transform appeared in \cite{Aluthge_IEOT_1990} for the first time.
In \cite{KOkubo2003} a more general notion called 
$t$-Aluthge transform, denoted by $\widetilde{T}_t$, and 
defined by $\widetilde{T}_t=|T|^{t}U|T|^{1-t}$ for $0 \leq t\leq 1$ was introduced.

The following inequality is evident.
\begin{align}\label{eq_00}
    \|\tilde{A}_{t}\|_2\leq \||A|^{t}U|A|^{1-t}\|_2\leq \||A|^{t}\|_2\||A|^{1-t}\|_2=\|A\|_2^{t}\|A\|_2^{1-t}=\|A\|_2.
\end{align}

The $t$-Aluthge transform coincides with the usual Aluthge transform for $t=\frac{1}{2}$.
When $t=1$, the operator $\widetilde{T}_1=|T|U$ is called the 
Duggal transform of $T\in \mathfrak{L}(\mathcal{H})$.  In
\cite{KShebrawiMBakherad2018}, the generalized Aluthge transform of the
operator $T$, denoted by $\widetilde{T}_{f,g},$ was introduced. It is defined by
$\widetilde{T}_{f,g}=f(|T|)Ug(|T|)$, where $f, g$ are non-negative
continuous functions such that $f(x)g(x)=x$ $(x\geq 0)$ and $T\in
\mathfrak{L}(\mathcal{H})$. The Aluthge transform has appeared in many results treating the numerical radius, as seen in \cite{Alomari_GJM_2024, Kittaneh_LAMA, Sheybani_Vietnam_2023, TYamazaki2007}.
Yamazaki in \cite{TYamazaki2007} proved that if $T\in\mathfrak{L}(\mathcal{H})$, then
\begin{align}\label{eq_1.8}
    w(T)\leq \frac{1}{2}(\|T\|+w(\widetilde{T})).
\end{align}
The author's of \cite{Amer_Kittaneh_Studia_2013} refine the inequality \eqref{eq_1.8} in a following way 
\begin{align}
     w(T)\leq \frac{1}{2}(\|T\|+\min_{0\leq t\leq 1}w(\widetilde{T}_t)).
\end{align}
Recently, the authors of \cite{Frakis_Kittaneh_Soltani_LAMA_2025} developed some new inequalities and equalities for the $p$-numerical radius using $t$-Aluthge transform. Further, they obtained a related Yamazaki-type inequality involving $p$-numerical radius. 

The main objective of this paper is to obtain several $p$-numerical radius inequalities via the $(f,g)$-Aluthge transform of Hilbert space operators and operator matrices. We also obtain some bounds which improve earlier well-known results.

Davidson and Power \cite{Davidson_Power_1986} proved that if $T$ and $S$ are positive operators in $\mathfrak{L}(\mathcal{H})$, then
\begin{align}\label{}
       \|T+S\|\leq \max\{\|T\|, \|S\|\}+\|TS\|^{1/2}.
    \end{align}
A refinement of this
inequality has been established in the following lemma.
\begin{lemma}\cite{Kittaneh_JFA_1997}
    Let $T, S$ are positive operators in $\mathfrak{L}(\mathcal{H})$. Then
    \begin{align}\label{Ineq_Kittaneh_JFA_1997}
       \|T+S\|\leq \max\{\|T\|, \|S\|\}+\|T^{1/2}S^{1/2}\|.
    \end{align}
\end{lemma}

The inequality \eqref{Ineq_Kittaneh_JFA_1997} has been generalized in \cite{Abu-Omar_Kittaneh_2015}.  
\begin{lemma}\cite{Abu-Omar_Kittaneh_2015}
If $T$ and $S$ are operators
in $\mathfrak{L}(\mathcal{H})$, then
   \begin{align}\label{Eq_1.12}
     \|T+S^*\| \leq \max\bigg\{\|S\|,\|T\|\bigg\}+\max\bigg\{\||S|^{\frac{1}{2}}|T^*|^{\frac{1}{2}}\|,\||T|^{\frac{1}{2}}|S^*|^\frac{1}{2}\|\bigg\}.
 \end{align}
\end{lemma}
The refinement of inequality \eqref{Eq_1.12} can be stated
as follows.
\begin{lemma}\cite{Shebrawi_LAA_2017}
If $T$ and $S$ are operators
in $\mathfrak{L}(\mathcal{H})$, then
   \begin{align}\label{Eq_1.13}
     \|T+S^*\| \leq \max\bigg\{\|S\|,\|T\|\bigg\}+\frac{1}{2}\bigg(\||S|^{\frac{1}{2}}|T^*|^{\frac{1}{2}}\|+\||T|^{\frac{1}{2}}|S^*|^\frac{1}{2}\|\bigg).
 \end{align}
\end{lemma}

\begin{lemma}\cite{Kittaneh_LAA_1992}\label{Lemma_1}
    Let $T, S\in\mathfrak{C}_p(\mathcal{H})$ be such that $TS$ is self-adjoint and $1\leq p\leq \infty$.
    Then
    $$\|TS\|_p\leq \|\Re(ST)\|_p.$$
\end{lemma}
\begin{lemma}\cite{Bottazzi_Conde_OAm_2021}\label{Lemma_2}
    Let $T, S\in\mathfrak{C}_p(\mathcal{H})$. 
    Then
    $$2^{\frac{-1}{p}}\|T\|_p\leq w_p(T)\leq \|T\|_p,~~~~~~1\leq p\leq 2,$$
    and 
    $$2^{\frac{1}{p}-1}\|T\|_p\leq w_p(T)\leq \|T\|_p,~~~~~~2\leq p\leq \infty.$$
\end{lemma}
\section{Main results}

\begin{theorem}\label{Thm_1}
    Let $T\in\mathfrak{C}_p(\mathcal{H})$, let $1\leq p\leq \infty$. If $f, g$ are non-negative continuous functions on $[0, \infty)$ such that $f(x)g(x)=x(x\geq 0)$. Then
    \begin{align*}
        w_p(T)\leq 2^{\frac{1}{p}-1}w_p(\widetilde{T}_{f, g})+2^{\frac{1}{p}-2}\|f^2(|T|)+g^2(|T|)\|_p.
    \end{align*}
\end{theorem}

\begin{proof}
    Let $T=U|T|$ be the polar decomposition of $T$. Then for any $\theta\in \mathbb{R}$, we have
    \begin{align*}
        \|\Re(e^{i\theta}T)\|_p&=\frac{1}{2}\|e^{i\theta}T+e^{-i\theta}T^*\|_p\\
        &=\frac{1}{2}\|e^{i\theta}U|T|+e^{-i\theta}|T|U^*\|_p\\
         &=\frac{1}{2}\|e^{i\theta}Ug(|T|)f(|T|)+e^{-i\theta}f(|T|)g(|T|)U^*\|_p\\
         &=\frac{1}{2}\left\|\begin{bmatrix}
             e^{i\theta}Ug(|T|) & f(|T|)\\
             0 & 0
         \end{bmatrix}\begin{bmatrix}
            f(|T|) & 0\\
             e^{-i\theta}g(|T|)U^* & 0
         \end{bmatrix}\right\|_p\\
         &\leq \frac{1}{2}\left\|\Re\left(\begin{bmatrix}
            f(|T|) & 0\\
             e^{-i\theta}g(|T|)U^* & 0
         \end{bmatrix}\begin{bmatrix}
             e^{i\theta}Ug(|T|) & f(|T|)\\
             0 & 0
         \end{bmatrix}\right)\right\|_p\\
         &= \frac{1}{2}\left\|\Re\left(\begin{bmatrix}
            e^{i\theta}f(|T|)Ug(|T|) & f^2(|T|)\\
            g^2(|T|) & e^{-i\theta}g(|T|)U^*f(|T|) 
         \end{bmatrix}\right)\right\|_p\\
         &= \frac{1}{2}\left\|\begin{bmatrix}
            \Re(e^{i\theta}\widetilde{T}_{f, g}) & \frac{f^2(|T|)+ g^2(|T|)}{2}\\
           \frac{f^2(|T|)+ g^2(|T|)}{2} & \Re(e^{-i\theta}(\widetilde{T}_{f, g})^*) 
         \end{bmatrix}\right\|_p\\
          &\leq  \frac{1}{2}\left\|\begin{bmatrix}
            \Re(e^{i\theta}\widetilde{T}_{f, g}) & 0\\
           0 & \Re(e^{-i\theta}(\widetilde{T}_{f, g})^*) 
         \end{bmatrix}\right\|_p + \frac{1}{2}\left\|\begin{bmatrix}
           0 & \frac{f^2(|T|)+ g^2(|T|)}{2}\\
           \frac{f^2(|T|)+ g^2(|T|)}{2} & 0 
         \end{bmatrix}\right\|_p\\
         &=2^{\frac{1}{p}-1}\|\Re(e^{i\theta}\widetilde{T}_{f, g})\|_p+2^{\frac{1}{p}-2}\|f^2(|T|)+g^2(|T|)\|_p.
    \end{align*}
    where the first inequality follows from the Lemma \ref{Lemma_1}.
    Taking supremum over $\theta\in \mathbb{R}$ both sides in the above inequality, we obtain our desired result.  
\end{proof}
As a special case for our Theorem \ref{Thm_1}, we have the following corollary, which is already proved in \cite[Theorem 2.3]{Frakis_Kittaneh_Soltani_LAMA_2025}.
\begin{corollary}\label{Cor_1}
    Let $T\in\mathfrak{C}_p(\mathcal{H})$, let $1\leq p\leq \infty$. If $f(x)=x^{1-t}$ , $g(x)=x^t$, $t\in [0, 1]$. Then
    \begin{align*}
        w_p(T)\leq 2^{\frac{1}{p}-1}\min_{0\leq t\leq 1}w_p(\widetilde{T}_{1-t})+2^{\frac{1}{p}-2}\||T|^{2(1-t)}+|T|^{2t}\|_p.
    \end{align*}
\end{corollary}
For $p=\infty$, we obtain the following inequalities.
\begin{remark}
    \begin{enumerate}
        \item [(i)] Let $T\in \mathfrak{L}(\mathcal{H})$  and for $p=\infty$ to the Theorem \ref{Thm_1}, we have the following inequality, , which is already proved by \cite[Corollary 2.4]{KShebrawiMBakherad2018}.
       \begin{align*}
        w(T)\leq \frac{1}{2}w(\widetilde{T}_{f, g})+\frac{1}{4}\|f^2(|T|)+g^2(|T|)\|.
    \end{align*}
    \item [(ii)] Let $T\in \mathfrak{L}(\mathcal{H})$  and for $p=\infty$ to the Corollary  \ref{Cor_1}, we have the following inequality (see \cite[Corollary 2.4]{Frakis_Kittaneh_Soltani_LAMA_2025}).
    \begin{align*}
        w(T)\leq \frac{1}{2}\min_{0\leq t\leq 1}w(\widetilde{T}_{1-t})+\frac{1}{4}\||T|^{2(1-t)}+|T|^{2t}\|.
    \end{align*}
    \end{enumerate}
\end{remark}
Similarly, for $f(x)=x^{1-t}$ , $g(x)=x^t$, $t\in [0, 1]$, we have the following inequalities.

\begin{align*}
        w_p(T)\leq 2^{\frac{1}{p}-1}\min_{0\leq t\leq 1}w_p(\widetilde{T}_{t})+2^{\frac{1}{p}-2}\||T|^{2t}+|T|^{2(1-t)}\|_p.
    \end{align*}
As some special cases to our results, we have some remarks, already established in \cite{Frakis_Kittaneh_Soltani_LAMA_2025}. 
\begin{remark}\label{Remark_2}
    \begin{enumerate}
        \item [(i)] Let $T\in\mathfrak{C}_p(\mathcal{H})$, let $1\leq p\leq \infty$ and for $t=\frac{1}{2}$. Then
        \begin{align*}
        w_p(T)\leq 2^{\frac{1}{p}-1} \left(w_p(\widetilde{T})+\|T\|_p\right).
        \end{align*}
         \item[(ii)] Using Lemma \ref{Lemma_2}, we have the following equality.
    For $T\in\mathfrak{C}_p(\mathcal{H})$, let $2\leq p\leq \infty$. If $\widetilde{T}=0$, then
    $$w_p(T)=2^{\frac{1}{p}-1}\|T\|_p.$$
    \end{enumerate}
    \end{remark}
    Here are some special cases to our results, which is already established in \cite{TYamazaki2007}.
    \begin{remark}
        For $p=\infty$ and $p=2$ in Remark \ref{Remark_2}, we obtain
        \begin{enumerate}
        \item [(i)] Let $T\in\mathfrak{L}(\mathcal{H})$. Then
        \begin{align*}
        w(T)\leq \frac{1}{2} \left(w(\widetilde{T})+\|T\|\right).
        \end{align*}
         \item[(ii)] 
    $$w(T)=\frac{\|T\|}{2}.$$
     \item[(iii)] $T\in\mathfrak{C}_2(\mathcal{H})$, then 
     \begin{align*}
        w_2(T)\leq \frac{1}{\sqrt{2}} \left(w_2(\widetilde{T})+\|T\|_2\right).
        \end{align*}
    \end{enumerate}
    \end{remark}

\begin{theorem}\label{Thm_2}
    Let $T\in\mathfrak{C}_p(\mathcal{H})$, let $2\leq p\leq \infty$. Then
    \begin{align*}
        w_p^2(T)\leq \frac{1}{4}\left(\|g(|T|)\widetilde{T}_{f, g}f(|T|)\|_{\frac{p}{2}}+\|f(|T|)(\widetilde{T}_{f, g})^*g(|T|)\|_{\frac{p}{2}}+\|T^*T+TT^*\|_{\frac{p}{2}}\right).
    \end{align*}
\end{theorem}

\begin{proof}
    Let $T=U|T|$ be the polar decomposition of $T$. Then for any $\theta\in \mathbb{R}$, we have
    \begin{align*}
        \left(\Re(e^{i\theta} T)\right)^2&=\left(\frac{e^{i\theta}T+e^{-i\theta}T^*}{2}\right)^2\\
        &=\frac{1}{4}\left(e^{2i\theta}T^2+e^{-2i\theta}(T^*)^2+T^*T+TT^*\right)\\
        &=\frac{1}{2}\left(e^{2i\theta}U|T|U|T|+e^{-2i\theta}|T|U^*|T|U^*+T^*T+TT^*\right)\\
        &=\frac{1}{4}\left(e^{2i\theta}Ug(|T|)f(|T|)Ug(|T|)f(|T|+e^{-2i\theta}f(|T|)g(|T|)U^*f(|T|)g(|T|)U^*+T^*T+TT^*\right)\\
         &=\frac{1}{4}\left(e^{2i\theta}Ug(|T|)\widetilde{T}_{f,g}f(|T|+e^{-2i\theta}f(|T|)(\widetilde{T}_{f,g})^*g(|T|)U^*+T^*T+TT^*\right).
    \end{align*}
    Thus,
    \begin{align*}
        \|(\Re(e^{i\theta} T)^2\|_{\frac{p}{2}}&\leq \frac{1}{4}\left(\|Ug(|T|)\widetilde{T}_{f,g}f(|T|\|_{\frac{p}{2}}+\|f(|T|)(\widetilde{T}_{f,g})^*g(|T|)U^*\|_{\frac{p}{2}}+\|T^*T+TT^*\|_{\frac{p}{2}}\right)\\
        &\leq \frac{1}{4}\left(\|g(|T|)\widetilde{T}_{f,g}f(|T|\|_{\frac{p}{2}}+\|f(|T|)(\widetilde{T}_{f,g})^*g(|T|)\|_{\frac{p}{2}}+\|T^*T+TT^*\|_{\frac{p}{2}}\right)~~(\mbox{by~\eqref{Eq_2.2}}).
    \end{align*}
    By taking the supremum in the above inequality over $\theta\in \mathbb{R}$, we get
     \begin{align*}
        w_p^2(T)\leq \frac{1}{4}\left(\|g(|T|)\widetilde{T}_{f, g}f(|T|)\|_{\frac{p}{2}}   +\|f(|T|)(\widetilde{T}_{f, g})^*g(|T|)\|_{\frac{p}{2}}+\|T^*T+TT^*\|_{\frac{p}{2}}\right).
    \end{align*}
\end{proof}
In \cite[Theorem 2]{Kittaneh_LMS_1993}, it was shown that for $A, B$ are positive operators in $\mathfrak{L}(\mathcal{H})$ and $X\in \mathcal{B(H)}$,
		the inequality
		\begin{align}\label{Heinz_ineq_0}
			\vertiii{A^\nu XB^{1-\nu}}\leq \vertiii{AX}^\nu\vertiii{XB}^{1-\nu},
		\end{align}
		holds for every $\nu\in [0, 1]$.
  
\begin{remark}
    By letting $g(x)=x^{1-t}, f(x)=x^t$, $0\leq t\leq 1$, to the above theorem, we have the following inequality, which is already established recently in \cite[Theorem 2.10]{Frakis_Kittaneh_Soltani_LAMA_2025}.
    \begin{enumerate}
        \item [(i)]  \begin{align*}
        w_p^2(T)&\leq \frac{1}{4}\left(\||T|^{1-t}\widetilde{T}_{t}|T|^t\|_{\frac{p}{2}}+\||T|^t(\widetilde{T}_{t})^*|T|^{1-t}\|_{\frac{p}{2}}+\|T^*T+TT^*\|_{\frac{p}{2}}\right)\\
        &\leq \frac{1}{4}\left(\||T|\widetilde{T}_{t}\|_{\frac{p}{2}}^{1-t}\|\widetilde{T}_{t}|T|\|_{\frac{p}{2}}^t+\||T|(\widetilde{T}_{t})^*\|_{\frac{p}{2}}^t\|(\widetilde{T}_{t})^*|T|\|_{\frac{p}{2}}^{1-t}+\|T^*T+TT^*\|_{\frac{p}{2}}\right)~(\mbox{by}~ \eqref{Heinz_ineq_0})\\
        &\leq \frac{1}{2}\|T\|_{\frac{p}{2}}\min_{0\leq t\leq 1}\|\widetilde{T}_{t}\|
        +\frac{1}{4}\|T^*T+TT^*\|_{\frac{p}{2}}, 
    \end{align*}
    for $2\leq p\leq \infty$.
    \item [(ii)] If we set $\widetilde{T}_{f,g}=0$ in Theorem \ref{Thm_2}. Then we obtain an inequality already established in \cite[Remark 2.11]{Frakis_Kittaneh_Soltani_LAMA_2025}.
    $$w_p^2(T)\leq \frac{1}{4}\|T^*T+TT^*\|_{\frac{p}{2}}.$$
    \end{enumerate}
    By using \cite[(18)]{Benmakhlouf_Hirzallah_Kittaneh_LAMA_2021}, we get  
    $$w_p^2(T)= \frac{1}{4}\|T^*T+TT^*\|_{\frac{p}{2}}.$$
   \end{remark}
   The following lemma is essential for our analysis in order to obtain the immediate corollary.
\begin{lemma}\cite{Kittaneh_JFA_1997}\label{Lemma_2.3}
    Let $T, S\in\mathfrak{C}_p(\mathcal{H})$ be positive, and let $p\geq 1$. 
    Then
    \begin{align}
        \|T+S\|_p\leq \left(\|T\|_p^p+\|S\|_p^p\right)^{\frac{1}{p}}+2^{\frac{1}{p}}\|T^{1/2}S^{1/2}\|_p.
    \end{align}
\end{lemma}
\begin{corollary}\label{Cor_2.2}
     Let $T\in\mathfrak{C}_p(\mathcal{H}), 2\leq p\leq \infty$. Then
     \begin{align*}
         w_p^2(T)\leq \frac{1}{4}\left(\|g(|T|)\widetilde{T}_{f, g}f(|T|)\|_{\frac{p}{2}}+\|f(|T|)(\widetilde{T}_{f, g})^*g(|T|)\|_{\frac{p}{2}}\right)+2^{\frac{2}{p}-2}(\|T\|_p^2+\|T^2\|_{\frac{p}{2}}).
     \end{align*}
\end{corollary}
\begin{proof}
    Using Theorem \ref{Thm_2}, and Lemma\ref{Lemma_2.3}, we obtain
    \begin{align*}
        w_p^2(T)&\leq \frac{1}{4}\left(\|g(|T|)\widetilde{T}_{f, g}f(|T|)\|_{\frac{p}{2}}+\|f(|T|)(\widetilde{T}_{f, g})^*g(|T|)\|_{\frac{p}{2}}+\|T^*T+TT^*\|_{\frac{p}{2}}\right)\\
        & \leq \frac{1}{4}\left(\|g(|T|)\widetilde{T}_{f, g}f(|T|)\|_{\frac{p}{2}}+\|f(|T|)(\widetilde{T}_{f, g})^*g(|T|)\|_{\frac{p}{2}}\right)+\frac{1}{4}\left\{\left(\||T|^2\|_{\frac{p}{2}}^{\frac{p}{2}}+\||T^*|^2\|_{\frac{p}{2}}^{\frac{p}{2}}\right)^{\frac{2}{p}}\right.\\
        &\left. \quad\hspace{9.0cm}+2^{2/p}\||T||T^*|\|_{\frac{p}{2}}\right\}\\
        &= \frac{1}{4}\left(\|g(|T|)\widetilde{T}_{f, g}f(|T|)\|_{\frac{p}{2}}+\|f(|T|)(\widetilde{T}_{f, g})^*g(|T|)\|_{\frac{p}{2}}\right)+2^{\frac{2}{p}-2}\left(\|T\|_p^2+\|T^2\|_{\frac{p}{2}}^2\right).
    \end{align*}
\end{proof}

\begin{remark}
    If $f(x)=x^t, g(x)=x^{1-t}, 0\leq t\leq 1$, to the Corollary \ref{Cor_2.2}, we obtain
    \begin{align*}
       w_p^2(T)\leq   \frac{1}{2}\|T\|_{\frac{p}{2}}\min_{0\leq t\leq 1}\|\widetilde{T_t}\|+2^{\frac{2}{p}-2}\left(\|T\|_p^2+\|T^2\|_{\frac{p}{2}}\right),
    \end{align*}
for every $T\in\mathfrak{C}_p(\mathcal{H}), 2\leq p\leq \infty$,    which is already in \cite[Corollary 2.13]{Frakis_Kittaneh_Soltani_LAMA_2025}. 
\end{remark}

\begin{theorem}\label{Thm3}
 Let $T, S\in \mathfrak{C}_p(\mathcal{H})$, and let $1\leq p\leq \infty$. Then
 \begin{align*}
     w_p\left(\begin{bmatrix}
         0 &T\\
         S & 0
     \end{bmatrix}\right)&\leq 2^{\frac{2}{p}-2}\bigg(\|f(|S|)g(|T^*|)\|_p+\|f(|T|)g(|S^*|)\|_p\bigg)\\
        & +2^{\frac{1}{p}-2}\bigg(\|f^2(|S|)+g^2(|S|)\|_p^p+\|f^2(|T|)+g^2(|T|)\|_p^p\bigg)^{\frac{1}{p}}.
 \end{align*}
\end{theorem}
\begin{proof}
    Let $T=U|T|$ and $S=V|S|$ be the polar decomposition of $T$ and $S$ respectively and let $\mathbb{T}=\begin{bmatrix}
        0 & T\\
        S & 0
    \end{bmatrix}$. Now from the polar decomposition of $\mathbb{T}=\begin{bmatrix}
        0 & U\\
        V & 0
    \end{bmatrix}\begin{bmatrix}
        |S| & 0\\
        0 & |T|
    \end{bmatrix}$ that 
    \begin{align*}        \widetilde{T_{f,g}}&=f(|T|)\begin{bmatrix}
        0 & U\\
        V & 0
    \end{bmatrix}g(|T|)\\
    &=\begin{bmatrix}
        f(|S|) & 0\\
        0 & f(|T|)
    \end{bmatrix}\begin{bmatrix}
        0 & U\\
        V & 0
    \end{bmatrix}\begin{bmatrix}
        g(|S|) & 0\\
        0 & g(|T|)
    \end{bmatrix} \\
    &=\begin{bmatrix}
        0 & f(|S|)Ug(|T|)\\
        f(|T|)Vg(|S|) & 0
    \end{bmatrix}
    \end{align*}
    Now using the fact $|T^*|^2=TT^*=U|T|^2U^*$, $|S^*|^2=SS^*=V|S|^2V^*$, so we have $g(|T|)=U^*g(|T^*|)U$, $g(|S|)=V^*g(|S^*|)V$ for every non-negative continuous function $g$ on $[0, \infty)$.
    So, by using a unitary operator $\mathbb{W}=\begin{bmatrix}
        0 & I\\
        I & 0
    \end{bmatrix}$, we obtain
    \begin{align}\label{Eq_2.4}
     w_p(\widetilde{T_{f,g}})&=   w_p\left(\begin{bmatrix}
        0 & f(|S|)Ug(|T|)\\
        f(|T|)Vg(|S|) & 0
    \end{bmatrix}\right)\nonumber\\
    &\leq  w_p\left(\begin{bmatrix}
        0 & f(|S|)Ug(|T|)\\
        0& 0
    \end{bmatrix}\right)+ w_p\left(\begin{bmatrix}
        0 & 0\\
        f(|T|)Vg(|S|) & 0
    \end{bmatrix}\right)\nonumber\\
    &=  w_p\left(\begin{bmatrix}
        0 & f(|S|)Ug(|T|)\\
        0& 0
    \end{bmatrix}\right)+ w_p\left(\mathbb{W}^*\begin{bmatrix}
        0 & f(|T|)Vg(|S|)\\
        0 & 0    \end{bmatrix}\mathbb{W}\right)\nonumber\\
        &=  w_p\left(\begin{bmatrix}
        0 & f(|S|)Ug(|T|)\\
        0& 0
    \end{bmatrix}\right)+ w_p\left(\begin{bmatrix}
        0 & f(|T|)Vg(|S|)\\
        0 & 0    \end{bmatrix}\right)\nonumber\\
        &=2^{\frac{1}{p}-1}\|f(|S|)Ug(|T|)\|_p+2^{\frac{1}{p}-1}\|f(|T|)Vg(|S|)\|_p~~~~~~\mbox{(by~\cite[Proposition 4.3]{Benmakhlouf_Hirzallah_Kittaneh_LAMA_2021})}\nonumber\\
      &=2^{\frac{1}{p}-1}\|f(|S|)UU^*g(|T^*|)U\|_p+2^{\frac{1}{p}-1}\|f(|T|)VV^*g(|S^*|)V\|_p\nonumber\\
      & \leq 2^{\frac{1}{p}-1}\bigg(\|f(|S|)g(|T^*|)\|_p+\|f(|T|)g(|S^*|)\|_p\bigg).
    \end{align}
    Now, using Theorem \ref{Thm_1} and \eqref{Eq_2.4}, we obtain
   \begin{align*}
        w_p(\mathbb{T})&\leq 2^{\frac{1}{p}-1}w_p(\widetilde{\mathbb{T}}_{f, g})+2^{\frac{1}{p}-2}\|f^2(|\mathbb{T}|)+g^2(|\mathbb{T}|)\|_p\\
        &=2^{\frac{1}{p}-1}\left(2^{\frac{1}{p}-1}\bigg(\|f(|S|)g(|T^*|)\|_p+\|f(|T|)g(|S^*|)\|_p\bigg)\right)\\
        & +2^{\frac{1}{p}-2}\left\|\begin{bmatrix}
            f^2(|S|)+g^2(|S|) & 0\\
            0 & f^2(|T|)+g^2(|T|)
        \end{bmatrix}\right\|_p\\
        &=2^{\frac{2}{p}-2}\bigg(\|f(|S|)g(|T^*|)\|_p+\|f(|T|)g(|S^*|)\|_p\bigg)\\
        & +2^{\frac{1}{p}-2}\bigg(\|f^2(|S|)+g^2(|S|)\|_p^p+\|f^2(|T|)+g^2(|T|)\|_p^p\bigg)^{\frac{1}{p}} ~\mbox{by \eqref{Eq_1.2}}.
    \end{align*}
\end{proof}

\begin{remark}
    If $f(x)=x^{1-t}$, $g(x)=x^t$ in Theorem \ref{Thm3}, then
    \begin{align*}
     w_p\left(\begin{bmatrix}
         0 &T\\
         S & 0
     \end{bmatrix}\right)&\leq 2^{\frac{2}{p}-2}\bigg(\||S|^{1-t}|T^*|^{t}\|_p+\||T|^{1-t}|S^*|^t\|_p\bigg)\\
        & +2^{\frac{1}{p}-2}\bigg(\||S|^{2(1-t)}+|S|^{2t}\|_p^p+\||T|^{2(1-t)}+|T|^{2t}\|_p^p\bigg)^{\frac{1}{p}}.
 \end{align*}
    For $t=\frac{1}{2}$
    \begin{align*}
     w_p\left(\begin{bmatrix}
         0 &T\\
         S & 0
     \end{bmatrix}\right)&\leq 2^{\frac{2}{p}-2}\bigg(\||S|^{\frac{1}{2}}|T^*|^{\frac{1}{2}}\|_p+\||T|^{\frac{1}{2}}|S^*|^\frac{1}{2}\|_p\bigg) +2^{\frac{1}{p}-2}\bigg(\||S|+|S|\|_p+\||T|+|T|\|_p\bigg)^{\frac{1}{p}}\\
        & =2^{\frac{2}{p}-2}\bigg(\||S|^{\frac{1}{2}}|T^*|^{\frac{1}{2}}\|_p+\||T|^{\frac{1}{2}}|S^*|^\frac{1}{2}\|_p\bigg)
         +2^{\frac{1}{p}-2}\bigg(2^{\frac{1}{p}}\bigg(\|S\|_p+\|T\|_p\bigg)^{\frac{1}{p}}\bigg)\\
         & =2^{\frac{2}{p}-2}\bigg(\||S|^{\frac{1}{2}}|T^*|^{\frac{1}{2}}\|_p+\||T|^{\frac{1}{2}}|S^*|^\frac{1}{2}\|_p         +\bigg(\|S\|_p^p+\|T\|_p^p\bigg)^{\frac{1}{p}}\bigg).
 \end{align*}
 \end{remark}

\begin{remark}
    For $p=\infty$ in Theorem \ref{Thm3}
    \begin{align}
         w_p\left(\begin{bmatrix}
         0 &T\\
         S & 0
     \end{bmatrix}\right) \leq \frac{1}{4} \bigg(\|f(|S|)g(|T^*|)\|+\|f(|T|)g(|S^*|)\|+\max\bigg\{\|f^2(|S|)+g^2(|S|)\|,\|f^2(|T|)+g^2(|T|)\|\bigg\}\bigg).
    \end{align}
\end{remark}

\begin{corollary}\label{Cor_2.3}
   Let $T, S\in \mathfrak{C}_p(\mathcal{H})$, and let $1\leq p<\infty$. Then
   \begin{align*}
       \|T+S^*\|_p &\leq 2^{\frac{1}{p}-1}\bigg(\|f(|S|)g(|T^*|)\|_p+\|f(|T|)g(|S^*|)\|_p\bigg)\\
        & +{\frac{1}{2}}\bigg(\|f^2(|S|)+g^2(|S|)\|_p^p+\|f^2(|T|)+g^2(|T|)\|_p^p\bigg)^{\frac{1}{p}}.
   \end{align*}
\end{corollary}

\begin{proof}
Let $\mathbb{T}=\begin{bmatrix}
        0 & T\\
        S & 0
    \end{bmatrix}$.
    By using \eqref{Eq_1.2}, we obtain
    \begin{align*}        2^{1/p}\|T+S^*\|_p&=\|\mathbb{T}+\mathbb{T}^*\|_p\\
        &\leq 2\max_{\theta\in \mathbb{R}}\|\Re(e^{i\theta}\mathbb{T})\|_p\\
        &=2w_p(\mathbb{T})\\
        &=2w_p\left(\begin{bmatrix}
            0& T\\
            S & 0
        \end{bmatrix}\right)\\
        &\leq 2^{\frac{2}{p}-1}\bigg(\|f(|S|)g(|T^*|)\|_p+\|f(|T|)g(|S^*|)\|_p\bigg)\\
        & +2^{\frac{1}{p}-1}\bigg(\|f^2(|S|)+g^2(|S|)\|_p^p+\|f^2(|T|)+g^2(|T|)\|_p^p\bigg)^{\frac{1}{p}}.
    \end{align*}
\end{proof}

\begin{remark}\label{Remark_2.8}
    If $f(x)=x^{1-t}$, $g(x)=x^t$ in the Corollary \ref{Cor_2.3}, we obtain
     \begin{align*}
    \|T+S^*\|_p&\leq 2^{\frac{1}{p}-1}\bigg(\||S|^{1-t}|T^*|^{t}\|_p+\||T|^{1-t}|S^*|^t\|_p\bigg)\\
        & +{\frac{1}{2}}\bigg(\||S|^{2(1-t)}+|S|^{2t}\|_p^p+\||T|^{2(1-t)}+|T|^{2t}\|_p^p\bigg)^{\frac{1}{p}}.
 \end{align*}
 For $t=\frac{1}{2}$,
 \begin{align*}
     \|T+S^*\|_p \leq 2^{\frac{1}{p}-1}\bigg(\||S|^{\frac{1}{2}}|T^*|^{\frac{1}{2}}\|_p+\||T|^{\frac{1}{2}}|S^*|^\frac{1}{2}\|_p      +\bigg(\|S\|_p^p+\|T\|_p^p\bigg)^{\frac{1}{p}}\bigg).
 \end{align*}
\end{remark}

\begin{remark}
   Letting $p=\infty$ in the Remark \ref{Remark_2.8}, we obtain \eqref{Eq_1.13}
     \begin{align*}
    \|T+S^*\|&\leq \frac{1}{2}\bigg(\||S|^{1-t}|T^*|^{t}\|+\||T|^{1-t}|S^*|^t\|\bigg)\\
        & +{\frac{1}{2}}\bigg(\max\bigg\{\||S|^{2(1-t)}+|S|^{2t}\|,\||T|^{2(1-t)}+|T|^{2t}\|\bigg\}\bigg).
 \end{align*}
 For $t=\frac{1}{2}$,
 \begin{align*}
     \|T+S^*\| \leq \frac{1}{2}\bigg(\||S|^{\frac{1}{2}}|T^*|^{\frac{1}{2}}\|+\||T|^{\frac{1}{2}}|S^*|^\frac{1}{2}\|\bigg)     +\max\bigg\{\|S\|,\|T\|\bigg\}.
 \end{align*}
\end{remark}

\begin{remark}
    From Remark \ref{Remark_2.8}
     \begin{align*}
     \|T+S^*\|_p \leq 2^{\frac{1}{p}-1}\bigg(\||S|^{\frac{1}{2}}|T^*|^{\frac{1}{2}}\|_p+\||T|^{\frac{1}{2}}|S^*|^\frac{1}{2}\|_p         +\bigg(\|S\|_p+\|T\|_p\bigg)^{\frac{1}{p}}\bigg).
 \end{align*}
Replace $S$ by $S^*$
\begin{align}\label{Eq_2.6}
     \|T+S\|_p \leq 2^{\frac{1}{p}-1}\bigg(\||S^*|^{\frac{1}{2}}|T^*|^{\frac{1}{2}}\|_p+\||T|^{\frac{1}{2}}|S|^\frac{1}{2}\|_p   +\bigg(\|S\|_p^p+\|T\|_p^p\bigg)^{\frac{1}{p}}\bigg).
 \end{align}
Also replace $S$ by $-S$ in \eqref{Eq_2.6}, we obtain
\begin{align}\label{Eq_2.7}
     \|T-S\|_p \leq 2^{\frac{1}{p}-1}\bigg(\||S^*|^{\frac{1}{2}}|T^*|^{\frac{1}{2}}\|_p+\||T|^{\frac{1}{2}}|S|^\frac{1}{2}\|_p   +\bigg(\|S\|_p^p+\|T\|_p^p\bigg)^{\frac{1}{p}}\bigg).
 \end{align}
 From \eqref{Eq_2.6} and \eqref{Eq_2.7}, we have 
 \begin{align}\label{Eq_2.8}
    \max\bigg\{ \|T+S\|_p, \|T-S\|_p\bigg\} \leq 2^{\frac{1}{p}-1}\bigg(\||S^*|^{\frac{1}{2}}|T^*|^{\frac{1}{2}}\|_p+\||T|^{\frac{1}{2}}|S|^\frac{1}{2}\|_p   +\bigg(\|S\|_p^p+\|T\|_p^p\bigg)^{\frac{1}{p}}\bigg).
     \end{align}
     The \eqref{Eq_2.8} is a refinement of \cite[Corollary 2.19]{Frakis_Kittaneh_Soltani_LAMA_2025}. 
\end{remark}

In particular if $T$ and $T^*$ are normal operator, then

\begin{align}\label{Eq_2.9}
    \max\bigg\{ \|T+S\|_p, \|T-S\|_p\bigg\} \leq 2^{\frac{1}{p}}\||T|^{\frac{1}{2}}|S|^\frac{1}{2}\|_p   + 2^{\frac{1}{p}-1}\bigg(\|S\|_p^p+\|T\|_p^p\bigg)^{\frac{1}{p}}.
     \end{align}

\begin{remark}
 Letting $p=\infty$ in \eqref{Eq_2.8}, we obtain
    \begin{align*}
         \max\bigg\{ \|T+S\|, \|T-S\|\bigg\} \leq {\frac{1}{2}}\bigg(\||S^*|^{\frac{1}{2}}|T^*|^{\frac{1}{2}}\|+\||T|^{\frac{1}{2}}|S|^\frac{1}{2}\|   +\max\{\|S\|,\|T\|\}\bigg).
    \end{align*}
    If $T$ and $S$ are normal operators, then we have
    \begin{align*}
         \max\bigg\{ \|T+S\|, \|T-S\|\bigg\} \leq \||T|^{\frac{1}{2}}|S|^\frac{1}{2}\|   +\frac{1}{2}\max\{\|S\|,\|T\|\}.
    \end{align*}
    If $T$ and $S$ are positive operators, then we have
    \begin{align*}
        \|T+S\| \leq \frac{1}{2}\max\{\|S\|,\|T\|\}+\|T^{\frac{1}{2}}S^{\frac{1}{2}}\|,
    \end{align*}
    which is a refinement of \eqref{Ineq_Kittaneh_JFA_1997}.
\end{remark}

%########################

\end{document}